\documentclass{article}

\usepackage[utf8]{inputenc}
\usepackage[dvipsnames]{xcolor}
\usepackage{amsmath}
\usepackage{amsthm}
\usepackage{amsfonts}

\theoremstyle{theorem}
\newtheorem{thm}{Theorem}
\newtheorem{prop}{Proposition}

\newtheorem{cor}{Corollary}

\newcommand{\mn}{\mathbb{N}}

\newcommand{\mr}{\mathbb{R}}
\newcommand{\mc}{\mathbb{C}}

\newcommand{\eq}[1]{\begin{equation*}#1\end{equation*}}
\newcommand{\eqn}[1]{\begin{equation}#1\end{equation}}
\newcommand{\aln}[1]{\begin{align}#1\end{align}}
\newcommand{\al}[1]{\begin{align*}#1\end{align*}}

\begin{document}
\title{A Class of Simple Rearrangements of the Alternating Harmonic Series}
\markright{Rearrangements of AHS}
\author{Maxim Gilula}

\maketitle


\begin{abstract}
We present an easily defined countable family of permutations of the natural numbers for which explicit rearrangements (i.e., the sums induced by the permutations) can be computed. The digamma function proves to be the key tool for the computations found here for the alternating harmonic series. The permutations $\phi$ under consideration are simple in a sense: they are involutions ($\phi\circ\phi$ is the identity function). We show that the countable set of rearrangements obtained from the simple involutions considered below are dense in the reals.
\end{abstract}

\section{Introduction.}
One can find Riemann's rearrangement theorem in many standard analysis textbooks, and even in some calculus textbooks (see e.g., \cite[Theorem 3.54]{babyrudin}). A sum $\sum_n a_n$ is said to \textit{conditionally converge} if $\sum_n a_n$ converges, but $\sum_n |a_n|$ diverges. Roughly, Riemann's rearrangement theorem states that given any conditionally convergent series $\sum_n a_n$ of reals $a_n$, and any real number $L$, there is a bijection $\phi$ from the natural numbers to themselves such that $\sum_n a_{\phi(n)}=L.$ Unfortunately, most bijections we can come up with do not give rearrangements that can be explicitly computed, while on the other hand, if we want a particular rearrangement, it is usually too difficult to write down an inducing bijection in closed form. Instead, they are usually defined via some recursive process similar to the one in the proof of Riemann's theorem. The goal of this paper is to present simple number-theoretic permutations that allow us to compute the induced rearrangements of the alternating harmonic series $\sum_{n=1}^\infty (-1)^{n+1}/n$. 

A permutation of order $n$ of a set $S$ is a bijection $\phi:S\to S$ such that $\phi^n=\phi\circ\cdots\circ\phi=\mathbf{1}_S$  and $n$ is the least such with this property, where $\circ$ denotes function composition and $\mathbf{1}_S$ is the identity function on $S$. In 1975, J. H. Smith\cite{jhs75} proved a refinement of Riemann's rearrangement theorem. The refinement implies, for example, that if we consider permutations of order 2 of the natural numbers, the same conclusion holds. That is, given $-\infty\le \alpha\le \beta\le\infty$ and a conditionally convergent series $\sum_n a_n$, there is a permutation $\phi$ of order 2 such that the partial sums $s_N=\sum_{n=1}^N a_{\phi(n)}$ satisfy
\[\liminf_{N\to\infty}s_N =\alpha,\text{ and }\limsup_{N\to\infty}s_N=\beta.\]
Other simple rearrangements have been considered elsewhere, e.g., by Stout\cite{stout} and the comprehensive references therein. Although questions about general rearrangements are very interesting, we instead consider an explicit family of permutations. Section 2 is devoted to studying the digamma function and an explicit class of permutations $\phi$ of order 2 defined by

\[ \phi (n) = \begin{cases} 
c\frac{n-b}{a}+d & \textrm{ if $n\equiv b\bmod a$} \\
a\frac{n-d}{c}+b & \textrm{ if $n\equiv d\bmod c$} \\
n  & \textrm{otherwise}.
\end{cases} \]
For example, consider $\phi$ as above with  $a=2, b=1, c=4, d=2.$ Then $\phi$ swaps odd integers with their double and fixes the remaining integers (those divisible by $4$). In general, whenever $\phi$ as above is a well-defined bijection, it induces a rearrangement of the alternating harmonic series that can be explicitly computed; for example, the above permutation that swaps odd integers with their double induces the rearrangement $\log(2)/4$. Although there are many other permutations of the alternating harmonic series for which one can find closed expressions, these seem to be particularly simple to concoct and write down explicitly. The countable family of such permutations is particularly interesting because it induces rearrangements of the alternating harmonic series that are dense in $\mr,$ which is the best we can hope for. This will be proved below.

Not many explicit rearrangements of even the alternating harmonic series are known, and even fewer rearrangements induced by permutations in closed form can be computed. Almost all examples of rearrangements we can write down come from a permutation defined with just an algorithm. For example, an exercise in Knopp's textbook on infinite series \cite[Ex. 51, p. 150]{knopp90} tells us that if we rearrange the alternating harmonic series by first summing the first $p_1$ consecutive positive terms then the first $p_2$ consecutive negative terms and repeating, then the sum converges to $\log(2)+\log(p_1/p_2)/2.$ \footnote{A solution can be found, for example, in \cite[p. 3]{agnes}. We provide a detailed proof for the case $p_1=1, p_2=4$ in Section 2.2 (C), and easily prove a generalization where we consider an arbitrary $N\ge 1$ and $p_i$ consecutive integers congruent to $i\bmod 2N$ for $1\le i\le 2N$, all thanks to the power of the digamma function.} There are many such exercises in the literature, but very few involve explicit forms for the permutation. As an example, define the permutation $\phi$ on positive integers by
\[ \phi (n) = \begin{cases} 
4\frac{n-1}{2}+2 & \textrm{ if $n\equiv 1\bmod 2$} \\
2\frac{n-2}{4}+1 & \textrm{ if $n\equiv 2\bmod 4$} \\
n  & \textrm{otherwise}.
\end{cases} \]
As noted in the preceding paragraph, the above is just an explicit way of writing down a bijection with the rule \lq\lq if $n$ is odd, double it; if it is a double of an odd number, divide it by two; otherwise, leave it fixed.\rq\rq \space Taking the sum of $(-1)^{\phi(n)+1}/\phi(n),$ the rearrangement is given by
\eq{-\frac{1}{2}+1-\frac{1}{6}-\frac{1}{4}-\frac{1}{10}+\frac{1}{3}-\frac{1}{14}-\frac{1}{8}...}
With some justification (justified below), one can show that this rearrangement converges to the same limit as the series
\eq{\sum_{n=0}^\infty \Bigg(-\frac{1}{2(4n+1)}+\frac{2}{4n+2}-\frac{1}{2(4n+3)}-\frac{1}{4n+4}\Bigg).}
Call this limit $L$. To compute $L$, we attempt to use what we know about the alternating harmonic series: by considering partial sums, $\log(2)$ must be equal to
\al{\lim_{N\to\infty}\sum_{n=0}^{4N}\frac{(-1)^n}{n+1}=&\lim_{N\to\infty}\sum_{n=0}^{2N}\Bigg(\frac{1}{2n+1}-\frac{1}{2n+2}\Bigg)\\ =& \lim_{N\to\infty}\sum_{n=0}^N\Bigg(\frac{1}{4n+1}-\frac{1}{4n+2}+\frac{1}{4n+3}-\frac{1}{4n+4}\Bigg).}
So by the linearity of limits, $L$ must be equal to
\al{\lim_{N\to\infty} \sum_{n=0}^{N} \Bigg(-\frac{1}{2(4n+1)}+\frac{2}{4n+2}&-\frac{1}{2(4n+3)}-\frac{1}{4n+4}\Bigg)\\
	 = \lim_{N\to\infty}-\frac{1}{2}\sum_{n=0}^N\Bigg(\frac{1}{4n+1}-&\frac{1}{4n+2}+\frac{1}{4n+3}-\frac{1}{4n+4}\Bigg)\\
	+\lim_{N\to\infty}\frac{3}{4}\sum_{n=0}^{N}\Bigg(\frac{1}{2n+1}-&\frac{1}{2n+2}=-\frac{1}{2}\log(2)+\frac{3}{4}\log(2)=\frac{\log(2)}{4}\Bigg).}
A more general way to compute $L$ is to play a game of considering the asymptotics of each summand separately. To succeed, it would help to know that $\lim_{N\to\infty}\sum_{n=1}^N 1/n -\log(N)\to \gamma$ as $N\to\infty,$ where $\gamma$ is Euler's constant; the reader should consult the references to see such examples. However, in this paper we avoid these kind of computations by considering the digamma function (and although it lets us kill off any trace of the gamma constant, this is not how this magical function got its name). 
	

\section{Explicit permutations of the alternating harmonic series.}

The main result is to explicitly compute rearrangements of the alternating harmonic series induced by a special class of permutations. Let $\mn$ denote the positive integers. We state the main result of this section:

\begin{thm} Assume $a,b,c,d \in\mn$ satisfy
\begin{itemize}
\item $b<a,$ 
\item $d<c,$ and
\item $\gcd(a,c)$ does not divide $d-b.$  
\end{itemize}
Define $\phi:\mn\to \mn$ by

\[ \phi (n) = \begin{cases} 
      c\frac{n-b}{a}+d & \textrm{ if\, $n\equiv b\bmod a$} \\
      a\frac{n-d}{c}+b & \textrm{ if\, $n\equiv d\bmod c$} \\
	  n  & \textrm{otherwise}.
   \end{cases} \]
Then $\phi$ is a permutation of the positive integers and 
\al{\sum_{n=1}^\infty\frac{(-1)^{\phi(n)+1}}{\phi(n)}&=\\
	\log(2)+&\log\Big(\frac{c}{a}\Big)\frac{(-1)^b+(-1)^{a+b}}{2a}+\log\Big(\frac{a}{c}\Big)\frac{(-1)^{d}+(-1)^{c+d}}{2c}.}
\end{thm}
\noindent The conditions $b<a, d<c$ are artificial and only help simplify the proof, but the divisibility condition on $d-b$ is necessary and sufficient for bijectivity.

We first discuss some basic facts about series and the digamma function.

\subsection{Basic facts and the digamma function}

We start with a result relating the convergence of the series $\sum_{i=0}^\infty d_i$ with the convergence of the series $\sum_{i=0}^\infty\sum_{j=0}^{n-1} d_{ni+j}.$ In particular, we can often take a conditionally convergent series and instead work with an absolutely convergent series by summing over these \lq\lq blocks.\rq\rq 

\begin{prop} Let $\{a_i\}_{i=0}^\infty$ be a sequence of reals converging to $0.$ Fix $n\in\mn$. $\sum_{i=0}^\infty a_i$ converges if and only if $\sum_{i=0}^\infty\sum_{j=0}^{n-1} a_{ni+j}$ converges. If the sums converge, they are equal.
\label{blocks}
\end{prop}

This follows from a standard result about Cauchy sequences in a metric space: if a subsequence converges, so does the sequence, and they converge to the same limit. Now we provide three key properties of the digamma function used in the proof of Theorem 1. First, we define the digamma function $\psi:(0,\infty)\to\mr$ by
	\eq{\psi(x):=\frac{d}{dx}\log\Gamma(x)=\frac{\Gamma'(x)}{\Gamma(x)},}
where one of many definitions of the standard gamma function $\Gamma(x)$ is given by
\eq{\Gamma(x):=\lim_{n\to\infty}\frac{(n-1)!n^x}{x(x+1)\cdots(x+n-1)}.}
This definition of $\Gamma(x)$ is particularly nice because it makes sense for all $x\in\mc$ except the nonpositive integers.

We reference a formula for the digamma function that is much more useful for our purposes:
\eq{\psi(x)=-\gamma+\sum_{n=0}^\infty \Bigg(\frac{1}{n+1}-\frac{1}{n+x}\Bigg).}
Here $\gamma$ is the Euler--Mascheroni constant (also called Euler's constant). This expression for digamma is deduced in \cite[Eq. (19), p. 138]{digamma1} from the definition above. This expression motivates the use of the digamma function in our calculations: the above is directly related to the harmonic series and the behavior of $\log(N)-\sum_{n=1}^N 1/n$ as $N$ grows to infinity. We summarize the properties of digamma needed for proving Theorem 1:

\begin{prop} Let $\psi$ be the digamma function for positive real values as defined above.
\begin{itemize}
\item[(i)] Let $a,b, N\in\mn.$ \eq{\sum_{n=0}^{N-1}\frac{1}{an+b}=\frac{\psi(b/a+N)-\psi(b/a)}{a}.} 
\item[(ii)] For all $s>0,$ 
\eq{\lim_{x\to\infty}\psi(x+s)-\log(x)=0.} 
\item[(iii)] Let $q$ be a positive rational and $m\in\mn.$ Then \eq{\displaystyle\sum_{j=0}^{m-1}\psi\Big(q+\frac{j}{m}\Big)=m(\psi(qm)-\log(m)).} 
\end{itemize}
\label{gamprop}
\end{prop}

\begin{proof}
	We deduce $(i)$ from \cite[Eq. (19$'''$), p. 139]{digamma1}. By equation (19$'''$) we know that for $x>0$,
	\eq{\psi(x+s)-\psi(x)=\sum_{n=0}^\infty\frac{s}{(n+x+s)(n+x)}=\sum_{n=0}^\infty\Bigg(\frac{1}{n+x}-\frac{1}{n+x+s}\Bigg).}
	By considering partial sums and letting $s=N\in\mn$ above, 
	\eq{\psi(x+N)-\psi(x)=\sum_{n=0}^{N-1}\frac{1}{n+x}.}
	Letting $x=b/a$ we get the desired equation.
	
	The proof of $(ii)$ can be found in the steps preceding \cite[Eq. (21), p. 139]{digamma1}.
	
	Finally, $(iii)$ can also be found in \cite[Eq. (23), p. 140]{digamma1}.
	
\end{proof}
We continue with one of the most necessary facts about explicit convergence of series for the proof of Theorem 1.
\begin{prop} For $1\le i\le n$ let $a_i, b_i\in \mn$ and $c_i\in\mr$ satisfy
	\eq{\frac{c_1}{a_1}+\cdots+\frac{c_n}{a_n}=0.}
Then the series
	\eq{\sum_{k=0}^\infty\Bigg(\frac{c_1}{a_1k+b_1}+\cdots+\frac{c_n}{a_nk+b_n}\Bigg)}
converges absolutely to \eq{-\sum_{i=1}^n \frac{c_i}{a_i}\psi\Big(\frac{b_i}{a_i}\Big).}
\end{prop}
\begin{proof}
Although absolute convergence holds, we do not use it for the rest of the article, so it is left as an exercise.

Next, we use property $(i)$ in Proposition 2 to get
\eq{\sum_{k=0}^{N-1}\sum_{i=1}^n \frac{c_i}{a_ik+b_i}=-\sum_{i=1}^n\frac{c_i}{a_i} \psi\Big(\frac{b_i}{a_i}\Big)+\sum_{i=1}^n\frac{c_i}{a_i}\psi\Big(\frac{b_i}{a_i}+N\Big).}
Therefore all we need to show is that 
\eq{\sum_{i=1}^n \frac{c_i}{a_i}\psi\Big(\frac{b_i}{a_i}+N\Big)\to 0 \text{ as } N\to\infty.}
But
\eq{\sum_{i=1}^n \frac{c_i}{a_i}\psi\Big(\frac{b_i}{a_i}+N\Big)=\sum_{i=1}^n\frac{c_i}{a_i}\left[\psi\Big(\frac{b_i}{a_i}+N\Big)-\log(N)\right]+\sum_{i=1}^n\frac{c_i}{a_i}\log(N).}
For all $N$ we know
\eq{\sum_{i=1}^n\frac{c_i}{a_i}\log(N)=\log(N)\sum_{i=1}^n\frac{c_i}{a_i}=\log(N)\cdot 0=0}
by the condition on the $c_i/a_i.$ Finally, by property $(ii)$ of Proposition 2,
\eq{\lim_{N\to\infty}\sum_{i=1}^n\frac{c_i}{a_i}\left[\psi\Big(\frac{b_i}{a_i}+N\Big)-\log(N)\right]=\sum_{i=1}^n \frac{c_i}{a_i} \lim_{N\to\infty} \left[\psi\Big(\frac{b_i}{a_i}+N\Big)-\log(N)\right]=0.}
\end{proof}
Another standard result about digamma function follows easily from this proposition:
\begin{cor}
	For all positive integers $N,$ \[\frac{1}{2N}\sum_{j=1}^{2N}(-1)^j \psi\left(\frac{j}{2N}\right)=\log(2).\]
\end{cor}
\begin{proof}
We can partition integers into consecutive integers of size $2N.$ Proposition 1, the alternating harmonic series is equal to \[\log(2)=\sum_{k=0}^\infty \sum_{j=1}^{2N} \frac{(-1)^{2Nk+j+1}}{2Nk+j}.\] By Prop 3 this equals
\[-\sum_{j=1}^{2N}\frac{(-1)^{j+1}}{2N}\psi\left(\frac{j}{2N}\right),\]
which is exactly the sum we want. 
\end{proof}
We now take our attention away from the main goal to argue why the digamma function is so nice when dealing with rearrangements of the alternating harmonic series.

\subsection{Examples and implications}
\begin{itemize}
	\item[(A)] If $a$ and $c$ are both odd, then 
	\eq{\log(2)+\log\Big(\frac{c}{a}\Big)\frac{(-1)^b+(-1)^{a+b}}{2a}+\log\Big(\frac{a}{c}\Big)\frac{(-1)^d+(-1)^{c+d}}{2c}=\log(2).}
	Although this may seem mysterious if we do not know how to obtain the explicit formula, there is a nice way to see the fact above by simply subtracting the partial sums of the rearranged series and the original series and showing that this difference tends to $0$ in the limit. Splitting the difference of the partial sums in a convenient way is a standard technique for dealing with permutations. So, we write
	
	\aln{  \sum_{n=1}^N  a_n - \sum_{n=1}^N& a_{\phi(n)}\nonumber\\
		=  \sum_{\phi(n)>N}^N a_n+& \sum_{\phi(n)\leq N}^N a_n -\sum_{\phi(n)>N}^N a_{\phi(n)}-\sum_{\phi(n)\leq N}^N a_{\phi(n)}\nonumber\\ 
		&\text{ } = \sum_{\phi(n)>N}^N (a_n-a_{\phi(n)})\label{inout}.}
	
	Letting $a,$ $c$ be odd, we can show the sum (\ref{inout}) tends to $0$ as $N\to\infty$, and therefore the permutation does not change the sum. This gives a class of explicit permutations one might expect \textit{a priori} to induce rearrangements of this series different from $\log(2)$.
	
	\item[(B)] Rearrangements of the alternating harmonic series given in Theorem 1 are dense in $\mr.$  Namely, given $a,b,c,d$ positive integers, the image of the function $f(a,b,c,d)$ defined by 
	\eq{f(a,b,c,d)=\log\Big(\frac{c}{a}\Big)\frac{(-1)^b+(-1)^{a+b}}{2a}+\log\Big(\frac{a}{c}\Big)\frac{(-1)^d+(-1)^{c+d}}{2c},\nonumber}
	restricted to the set where $a,b,c,d$ satisfy the conditions of Theorem 1, is dense in $\mr.$ In order to prove this, one needs to prove density of, e.g., 
	\eq{\frac{\log(n/m)}{m}+\frac{\log(m/n)}{n}}
	in $\mr$ for $n,m\in\mn.$ Once this is proved, choosing $a=6m+3,$ $b=1,$ $c=6n,$ and $d=2,$ (so $\gcd(a,c)\neq 1=d-b)$, one can show density of the rearrangements in $(0,\infty).$ Swapping $b$ and $d$ provides the same proof for $(-\infty, 0).$ This is just a matter of sorting out epsilons and deltas: the proof is elementary but may require a page of computations.
	
	\item[(C)] Now we give an example of a real number that is not equal to a rearrangement guaranteed by Theorem 1: analyzing parities of $a,b,c,d,$ we can deduce that no rearrangement induced by involutions as in the main theorem can equal $0.$ However, there are simple rearrangements of the alternating harmonic series that converge to $0$ described by an algorithm: consider the rearrangement defined by the first positive term followed by the first four consecutive negative terms, and repeat. That is, consider the series
	\eq{\textcolor{ForestGreen}{1}\textcolor{red}{-\frac{1}{2}-\frac{1}{4}-\frac{1}{6}-\frac{1}{8}}\textcolor{ForestGreen}{+\frac{1}{3}}\textcolor{red}{-\frac{1}{10}-\frac{1}{12}-\frac{1}{14}-\frac{1}{16}}\textcolor{ForestGreen}{+\frac{1}{5}}\textcolor{red}{-}...,}
	which converges to $0.$ We provide a standard computation here without using digamma function. By Proposition 1, this sum converges to the same limit as the sum
	\eq{\sum_{n=0}^\infty\Bigg(\frac{1}{2n+1}-\frac{1}{8n+2}-\frac{1}{8n+4}-\frac{1}{8n+6}-\frac{1}{8n+8}\Bigg).}
	The partial sums may be expressed as
	\al{ & \sum_{n=0}^N\Bigg(\frac{1}{2n+1}-\frac{1}{8n+2}-\frac{1}{8n+4}-\frac{1}{8n+6}-\frac{1}{8n+8}\Bigg)\\
		& =  \sum_{n=0}^N\frac{1}{2n+1}-\frac{1}{2}\sum_{n=0}^{N}\Bigg(\frac{1}{4n+1}+\frac{1}{4n+2}+\frac{1}{4n+3}+\frac{1}{4n+4}\Bigg)\\
		& =  \sum_{n=1}^{2N+1}\frac{1}{n}-\frac{1}{2}\sum_{n=1}^{N}\frac{1}{n} -\frac{1}{2}\sum_{n=1}^{4N+4}\frac{1}{n}.}
	Now we add and subtract $\log$ terms to exploit the behavior of the alternating harmonic series: the above equals
	\al{\Bigg(\sum_{n=1}^{2N+1}\frac{1}{n}-\log(2N+1)\Bigg)+&\frac{1}{2}\Bigg(\log(N)-\sum_{n=1}^{N}\frac{1}{n}\Bigg)\\+ \frac{1}{2}\Bigg(&\log(4N+4)-\sum_{n=1}^{4N+4}\frac{1}{n}\Bigg)+\log\Bigg(\frac{2N+1}{2\sqrt{N^2+N}}\Bigg).}
	As $N\to\infty,$ the above converges to $\gamma-\gamma/2-\gamma/2+\log(1)=0.$
	
	It is still open whether there are permutations expressible in closed form such that the induced rearrangement of the alternating harmonic series converges to $0.$ Computing this way does not feel satisfying.
	\item[(D)] Using Proposition 3 instead of \textit{ad hoc} methods in (C), we see right away that this rearrangement sums to \al{-\frac{1}{2}\psi\left(\frac{1}{2}\right)+ \sum_{j=0}^{4-1}\frac{1}{8}\psi\left(\frac{1}{4}+\frac{j}{4}\right)\stackrel{Prop. 2(iii)}{=}& \\ -\frac{1}{2}\psi\left(\frac{1}{2}\right)+ \frac{4}{8}\left[\psi\left(\frac{4}{4}\right)-\log(4)\right]&=
	  -\log(2)+\frac{\psi(1)-\psi(1/2)}{2}.} By Corollary 1 with $N=1$, this is just $-\log(2)+\log(2)=0.$ Using the same tools, we prove a  generalization of the sum $p_1$ positives followed by $p_2$ negatives, and repeat, which is equal to \[\textcolor{ForestGreen}{1+\frac{1}{3}+\cdots +\frac{1}{2p_1-1}}\textcolor{red}{-\frac{1}{2}-\cdots -\frac{1}{2p_2}}\textcolor{ForestGreen}{+}\cdots = \log(2)+\frac{1}{2}\log\left(\frac{p_1}{p_2}\right).\]
  
  		Claim: The sum of $p_1$ consecutive terms $1\bmod 2N$ plus the sum of $p_2$ consecutive terms $2\bmod 2N$ plus $\dots$ plus the sum of $p_{2N}$ consecutive terms $2N\bmod 2N$ is equal to \[\log(2)+\frac{1}{2N}\log\left(\frac{p_1p_3\cdots p_{2N-1}}{p_2p_4\cdots p_{2N}}\right).\]
  		
  		\begin{proof}
  			By Proposition 1, we can write the sum from the claim as \[\sum_{n=0}^\infty\sum_{i=1}^{2N}\sum_{j=0}^{p_i-1} \frac{(-1)^{i+1}}{2Np_in+2Nj+i}.\]
  			The inner most sum represents $p_i$ consecutive terms $i\bmod 2N,$ and the second sum is over each of the chosen equivalence classes $\bmod$ $2N;$ then we need to sum all of these.
  			
  			Verifying the hypothesis of Proposition 3 that \[\sum_{i=1}^{2N}\sum_{j=0}^{p_i-1}\frac{(-1)^{i+1}}{2Np_i}=\sum_{i=1}^{2N} \frac{(-1)^{i+1}}{2N}=0,\] it guarantees our sum converges to
  			\[-\sum_{i=1}^{2N}\sum_{j=0}^{p_i-1}\frac{(-1)^{i+1}}{2Np_i}\cdot\psi\left(\frac{2Nj+i}{2Np_i}\right).\] Next, Prop. 2$(iii)$ guarantees \[\sum_{j=0}^{p_i-1}\psi\left(\frac{2Nj+i}{2Np_i}\right)=\sum_{j=0}^{p_i-1}\psi\left(\frac{j}{p_i}+\frac{i}{2Np_i}\right)=p_i\left[\psi\left(\frac{i}{2N}\right)-\log(p_i)\right],\]
  			so that
  			\[\sum_{i=1}^{2N}\frac{(-1)^i}{2Np_i}\cdot p_i \left[\psi\left(\frac{i}{2N}\right)-\log(p_i)\right]=\]\[\frac{1}{2N}\sum_{i=1}^{2N}(-1)^i \psi\left(\frac{i}{2N}\right)+\frac{1}{2N}\sum_{i=1}^{2N}(-1)^{i+1}\log(p_i).\]
  			But Corollary 1 and basic properties of $\log$ simplify this to what we need:
  			\[\log(2)+\frac{1}{2N}\log\left(\frac{p_1p_3\cdots p_{2N-1}}{p_2p_4\cdots p_{2N}}\right).\]
  		\end{proof}
	The last example gives a flavor of the proof of Theorem 1. In some ways the proof of Theorem 1 is easier, and in some, harder. More detail is given below for how to think about summing over partitions.
\end{itemize}

\subsection{The digamma function and rearrangements of the alternating harmonic series}
Now we move on to our proof of Theorem 1.

\begin{proof}
Partitioning $\mn$ into sets of those elements that are $b\bmod a,$ $d\bmod c,$ and neither,  assuming $\gcd(a,c)$ does not divide $d-b,$ one can show $\phi$ is a well-defined bijection. 

We now split $\mn$ into equivalence classes modulo $2ac: \{2ack+1\},...,\{2ack+2ac\}.$ We choose $2ac$ because it is even and is divisible by $\gcd(a,c),$ two properties sufficient for capturing everything that is going on in the partial sums of our rearrangement. The philosophy behind computing the rearrangements is to consider partial sums of the original series of height $2acN$, see which terms are brought into the partial sum by the permutation, which are taken out, and which ones are not moved. Let's consider the partial sums of the alternating harmonic series and split the partial sums into blocks:
\eq{\sum_{k=1}^{2acN} \frac{(-1)^{k+1}}{k}=\sum_{k=0}^{N-1}\sum_{n=0}^{2ac-1}\frac{(-1)^{2ack+n}}{2ack+n+1}=\sum_{k=0}^{N-1}\sum_{n=0}^{2ac-1}\frac{(-1)^{n}}{2ack+n+1}.}
Now consider the blocks inside the series. We reorganize them  into a more convenient form for the purpose of explicitly seeing how $\phi$ changes the series, and note that we can consider each block separately in order to see what is going on with the series as a whole:
\al{&\sum_{n=0}^{2ac-1}\frac{(-1)^{n}}{2ack+n+1}= \sum_{n=0}^{2c-1}\frac{(-1)^{na+b+1}}{2ack+na+b}+\sum_{n=0}^{2a-1}\frac{(-1)^{nc+d+1}}{2ack+nc+d}\\
& +\Bigg(\sum_{n=0}^{2ac-1}\frac{(-1)^n}{2ack+n+1}-\sum_{n=0}^{2c-1}\frac{(-1)^{na+b+1}}{2ack+na+b}-\sum_{n=0}^{2a-1}\frac{(-1)^{nc+d+1}}{2ack+nc+d}\Bigg).
}
Out of the five sums, the first two sums are those that are moved away when we apply the permutation $\phi$, and the remaining terms are those that do not change. Note that for each $k$ there are $2c$ terms that are $b\bmod a$ in the set $\{2ack+n+1: 0\le n\le 2ac-1\}$:
\eq{2ack+b, 2ack+a+b, 2ack+2a+b,...,2ack+(2c-1)a+b,}
and $2a$ terms that are $d\bmod b$:
\eq{2ack+d, 2ack+c+d, 2ack+2c+d,...,2ack+(2a-1)c+d.}
When we apply $\phi,$ the first sum on the right of the equality containing $2c$ summands is exchanged with a sum of $2a$ terms by the rule
\eq{\phi(2ack+na+b)= 2c^2k+nc+d,}
the second sum of $2a$ summands will be exchanged with a sum of $2c$ terms by
\eq{\phi(2ack+nc+d)= 2a^2k+na+b,}
and the terms left over in the parentheses are exactly the ones that are fixed. Hence, the partial sum over $k$ of the rearranged series contains the terms

\al{&\sum_{n=0}^{2c-1}\frac{(-1)^{nc+d+1}}{2c^2k+nc+d}+\sum_{n=0}^{2a-1}\frac{(-1)^{na+b+1}}{2a^2k+na+b}\\
& +\Bigg(\sum_{n=0}^{2ac-1}\frac{(-1)^n}{2ack+n+1}-\sum_{n=0}^{2c-1}\frac{(-1)^{na+b+1}}{2ack+na+b}-\sum_{n=0}^{2a-1}\frac{(-1)^{nc+d+1}}{2ack+nc+d}\Bigg)\\
 = & \underbrace{\sum_{n=0}^{2c-1}\frac{(-1)^{nc+d+1}}{2c^2k+nc+d}}_{S_{(c,c)}(k)}+\underbrace{\sum_{n=0}^{2a-1}\frac{(-1)^{na+b+1}}{2a^2k+na+b}}_{S_{(a,a)}(k)}\\
& +\underbrace{\sum_{n=0}^{2c-1}\frac{(-1)^{na+b}}{2ack+na+b}}_{S_{(c,a)}(k)}+\underbrace{\sum_{n=0}^{2a-1}\frac{(-1)^{nc+d}}{2ack+nc+d}}_{S_{(a,c)}(k)} + \underbrace{\sum_{n=0}^{2ac-1}\frac{(-1)^n}{2ack+n+1}}_{AH(k)}.}
We have accounted for each term since $\phi$ is a bijection. It is now clear why we wanted the partial sum to be divisible by $2ac$: we move $2c$ and $2a$ terms out and bring $2c$ and $2a$ terms in, i.e., $2ac$ is divisible by $a$ and $c.$ Why we wanted an even number is pointed out later.

Next, we claim that the sum of the coefficients of $k$ in our new blocks is 0; if this is the case, we can consider a new infinite series that is absolutely convergent by Proposition 3, we compute the sum by the same proposition, and can finish the proof with property $(iii)$ of Proposition 2. We complete the proof by showing that the sum obtained is the same as the original rearranged sum by Proposition 1. This is simply because we can sum over blocks; together with the linearity property of sums for two convergent series, we know the sum is equal to the limit of the partial sums over each of the blocks. 

The following are true:
\eqn{\sum_{n=0}^{2a-1}\frac{(-1)^{na+b+1}}{2a^2}+\sum_{n=0}^{2c-1}\frac{(-1)^{na+b}}{2ac}=0,\label{(I)}}
\eqn{ \sum_{n=0}^{2c-1}\frac{(-1)^{nc+d+1}}{2c^2}+\sum_{n=0}^{2a-1}\frac{(-1)^{nc+d}}{2ac}=0,\label{(II)}}
and
\eqn{ \sum_{n=0}^{2ac-1}\frac{(-1)^n}{2ac}=0.\label{(III)}}
Indeed if $a$ is odd, then (\ref{(I)}) is zero since both sums are identically zero. Otherwise, the left hand side of (\ref{(I)}) is  \eq{\frac{2a(-1)^{b+1}}{2a^2}+\frac{2c(-1)^b}{2ac}=\frac{(-1)^{b+1}}{a}-\frac{(-1)^{b+1}}{a}=0.}
By completely symmetric computations, (\ref{(II)}) is $0$ as well. Finally, $(\ref{(III)})$ is just $0$. Now Proposition 3 may be applied since the sum of coefficients in (\ref{(I)}), (\ref{(II)}), and (\ref{(III)}) is zero. The proposition guarantees that the sum $\sum_{k=0}^\infty S_{(a,a)}(k)+S_{(c,a)}(k)$ converges to \eqn{\sum_{n=0}^{2a-1}\frac{(-1)^{na+b}}{2a^2}\psi\Big(\frac{na+b}{2a^2}\Big) + \sum_{n=0}^{2c-1}\frac{(-1)^{na+b+1}}{2ac}\psi\Big(\frac{na+b}{2ac}\Big),\label{I=}}
and by (\ref{(II)}), $\sum_{k=0}^\infty \Big(S_{(c,c)}(k)+S_{(a,c)}(k)\Big)$ converges to
\eqn{ \sum_{n=0}^{2c-1}\frac{(-1)^{nc+d}}{2c^2}\psi\Big(\frac{nc+d}{2c^2}\Big) + \sum_{n=0}^{2a-1}\frac{(-1)^{nc+d+1}}{2ac}\psi\Big(\frac{nc+d}{2ac}\Big).\label{II=}}
Finally, Proposition 1 tells us 
\eq{\sum_{k=0}^\infty AH(k)=\log(2)}
since we just took the limit of a subsequence of partial sums of the alternating harmonic series. Since each of these three series were convergent, we know that $(\ref{I=})+(\ref{II=})+\log(2)$ is the value of our rearranged series. So the final goal is to compute $(\ref{I=})+(\ref{II=}).$

Out of the four sums involving the digamma function that we need to compute, we compute one and note that the remaining computations are very similar. By splitting up the odd and even terms, one sees that
\aln{ \sum_{n=0}^{2c-1}\frac{(-1)^{nc+d}}{2c^2}\psi\Big(\frac{nc+d}{2c^2}\Big) & = \frac{(-1)^d}{2c^2}\sum_{n=0}^{2c-1}(-1)^{nc}\psi\Big(\frac{n}{2c}+\frac{d}{2c^2}\Big)\nonumber\\
 =\frac{(-1)^d}{2c^2}&\Bigg(\sum_{n=0}^{c-1}\psi\Big(\frac{n}{c}+\frac{d}{2c^2}\Big)+\sum_{n=0}^{c-1}(-1)^c\psi\Big(\frac{2n+1}{2c}+\frac{d}{2c^2}\Big)\Bigg).\label{star}
}
Now, (\ref{star}) equals
\eqn{
 \frac{(-1)^d}{2c^2}\Bigg(\sum_{n=0}^{c-1}\psi\Big(\frac{n}{c}+\frac{d}{2c^2}\Big)+(-1)^c\sum_{n=0}^{c-1}\psi\Big(\frac{n}{c}+\frac{d}{2c^2}+\frac{1}{2c}\Big)\Bigg).\label{starstar}
}
We apply identity $(iii)$ of the digamma function, namely
\eq{\sum_{j=0}^{m-1}\psi\left(q+\frac{j}{m}\right)=m(\psi(qm)-\log(m)),}
to see that (\ref{starstar}) is none other than
\al{ \frac{(-1)^d}{2c^2} & \Bigg(c\Big(\psi\Big(\frac{d}{2c}\Big)-\log(c)\Big)+ c(-1)^c \Big(\psi\Big(\frac{d}{2c}+\frac{1}{2}\Big)-\log(c)\Big)\Bigg)\\
& = \frac{(-1)^d}{2c}\Bigg(\psi\Big(\frac{d}{2c}\Big)+ (-1)^c\psi\Big(\frac{d}{2c}+\frac{1}{2}\Big)-\log(c) ((-1)^{c}+1)\Bigg).
}
By identical arguments, we can show that the other three series in (\ref{I=})+(\ref{II=}) equal
\aln{& \frac{(-1)^b}{2a}\Bigg(\psi\Big(\frac{b}{2a}\Big)+ (-1)^a\psi\Big(\frac{b}{2a}+\frac{1}{2}\Big)-\log(a) ((-1)^{a}+1)\Bigg),\label{star1}\\
&  \frac{(-1)^{b+1}}{2a}\Bigg(\psi\Big(\frac{b}{2a}\Big)+ (-1)^a\psi\Big(\frac{b}{2a}+\frac{1}{2}\Big)-\log(c) ((-1)^{a}+1)\Bigg),\label{star2}\\
&  \frac{(-1)^{d+1}}{2c}\Bigg(\psi\Big(\frac{d}{2c}\Big)+ (-1)^c\psi\Big(\frac{d}{2c}+\frac{1}{2}\Big)-\log(a) ((-1)^{c}+1)\Bigg).\label{star3}
}
Summing these \lq\lq simplifications," we see that all digamma terms cancel:
\al{(\ref{starstar})+(\ref{star1})+(\ref{star2})+(\ref{star3})&=\\
	\log\Big(\frac{c}{a}\Big)&\frac{(-1)^b+(-1)^{a+b}}{2a}+\log\Big(\frac{a}{c}\Big)\frac{(-1)^d+(-1)^{c+d}}{2c}.}
Adding the $\log(2)$ from the sum of $AH(k),$ we see that the rearrangement of the alternating harmonic series induced by $\phi$ converges to 
\eq{\log(2)+\log\Big(\frac{c}{a}\Big)\frac{(-1)^b+(-1)^{a+b}}{2a}+\log\Big(\frac{a}{c}\Big)\frac{(-1)^d+(-1)^{c+d}}{2c}.}
\end{proof}

\subsection{Permutations of order $m$ of the alternating harmonic series}

A slight generalization of Theorem 1, stated as a corollary because there are no new ideas in the proof, considers instead a class of permutations of order $m$:

\begin{cor} Let $a_i, b_i$ be positive integers with
\begin{itemize}
\item $b_i<a_i$ for $1\le i \le m,$
\item $\gcd(a_i, a_j)$ does not divide $b_i-b_j$ for distinct $1\le i,j\le m.$
\end{itemize}
If we permute the elements of the alternating harmonic series by \eq{a_1k+b_1\implies\cdots\implies a_{m-1}k+b_{m-1}\implies a_mk+b_m\implies a_1k+b_1,} i.e., if we define $\phi:\mn\to\mn$ by

\[ \phi (n) = \begin{cases} 
      a_2\frac{n-b_1}{a_1}+b_2 & \textrm{ if $n\equiv b_1\bmod a_1$} \\
	a_3\frac{n-b_2}{a_2}+b_3 & \textrm{ if $n\equiv b_2\bmod a_2$} \\
	&\vdots \\
    	  a_m\frac{n-b_{m-1}}{a_{m-1}}+b_m & \textrm{ if $n\equiv b_{m-1}\bmod a_{m-1}$} \\
	  a_1\frac{n-b_m}{a_m}+b_1 & \textrm{ if $n\equiv b_m\bmod a_m$} \\
	  n  & \textrm{otherwise},
   \end{cases} \]
then $\phi$ is a permutation of the positive integers and 

\eq{\sum_{n=1}^\infty \frac{(-1)^{\phi(n)+1}}{\phi(n)}=\log(2)+\frac{1}{2}\sum_{i=1}^{m}\log\Big(\frac{a_{i}}{a_{i+1}}\Big)\Bigg{[}\frac{(-1)^{b_{i+1}}+(-1)^{a_{i+1}+b_{i+1}}}{a_{i+1}}\Bigg{]},}
where $a_{m+1}=a_1$ and $b_{m+1}= b_1.$
\end{cor}
The proof is very similar to that of Theorem 1, but bulkier. We can apply the same techniques to compute rearrangements for conditionally convergent real series $\sum_n a_n$, where $a_n$ is equal to the quotient $P(n)/Q(n)$ of polynomials $P,Q$. Sometimes all the digamma terms cancel as in our case, but more often than not, they don't.

\section{Acknowledgments}
The author would very much like to thank Professor Michael J. Khoury for his encouragement and inspiration, as well as many examples of explicit rearrangements not found in the literature. (He also shared with me many interesting permutations of the alternating harmonic series that can be explicitly computed!) The author would also like to thank Soumyashant Nayak for insightful conversations. The referees were particularly helpful in terms of clarity of writing and examples, organization, and typos.

I would also like to thank Sarth Chavan from Euler Circle for encouraging me to post this on the arxiv.

\end{document}